\documentclass[reqno, letterpaper, oneside]{article}
\usepackage{amsmath,amsfonts,amssymb,amsthm}
\usepackage{comment}

\usepackage{bbm}
\usepackage{setspace}
\usepackage{geometry}
\usepackage{enumerate}

\usepackage[colorlinks=false,,linktoc=allbookmarksopen=true,linktocpage,pdftex]{hyperref}
\usepackage[pdftex]{bookmark}
\hypersetup{
	colorlinks=true,       
	linkcolor=blue,        
	citecolor=blue,        
	filecolor=magenta,     
	urlcolor=blue         
}

\usepackage{tabularx} 

\usepackage{mathtools}
\usepackage{cleveref}
\usepackage{enumitem}

\theoremstyle{plain}
\newtheorem{theorem}{Theorem}
\newtheorem{lemma}{Lemma}[section]
\newtheorem{proposition}[lemma]{Proposition}
\newtheorem{corollary}[lemma]{Corollary}

\theoremstyle{definition}

\newtheorem{definition}[lemma]{Definition}

\newcommand{\refT}[1]{Theorem~\ref{#1}}
\newcommand{\refC}[1]{Corollary~\ref{#1}}

\newcommand{\refL}[1]{Lemma~\ref{#1}}

\newcommand{\refS}[1]{Section~\ref{#1}}

\newcommand{\refPp}[1]{Proposition~\ref{#1}}

\newcommand{\refD}[1]{Definition~\ref{#1}}

\newcommand{\cB}{\mathcal{B}}

\newcommand{\cE}{\mathcal{E}}

\newcommand{\cN}{\mathcal{N}}

\newcommand{\lsim}{\lesssim}
\newcommand{\gsim}{\gtrsim}

\newcommand{\al}{\beta}
\newcommand{\incomp}{\rm Incomp}
\newcommand{\comp}{\rm Comp}
\newcommand{\spread}{\rm Spread}

\newcommand\lrset[1]{\ensuremath{\left\{#1\right\}}}

\newcommand\lrpar[1]{\left(#1\right)}

\newcommand\lrsqpar[1]{\left[#1\right]}

\newcommand\lrcpar[1]{\left\{#1\right\}}
\newcommand\abs[1]{|#1|}

\newcommand\biggabs[1]{\biggl|#1\biggr|}

\newcommand\ceil[1]{\lceil#1\rceil}

\DeclarePairedDelimiterX{\ip}[2]{\langle}{\rangle}{#1, #2}
\DeclarePairedDelimiterX{\norm}[1]{\lVert}{\rVert}{#1} 
\DeclarePairedDelimiterX{\size}[1]{\lvert}{\rvert}{#1} 
\newcommand{\bP}{\mathbb{P}}
\newcommand{\cF}{\mathcal{F}}
\newcommand{\bS}{\mathbb{S}}

\newcommand{\dist}{\text{dist}}

\renewcommand{\Pr}{\mathbb{P}}
\newcommand{\E}{\mathbb{E}}
\newcommand{\R}{\mathbb{R}}

\newcommand{\N}{\mathbb{N}}

\newcommand{\eps}{\varepsilon}

\newcommand{\one}{\textbf{1}}

\let\OLDthebibliography\thebibliography
\renewcommand\thebibliography[1]{
  \OLDthebibliography{#1}
  \setlength{\parskip}{0pt}
  \setlength{\itemsep}{0pt plus 0.3ex}
}

\title{The smallest singular value of inhomogenous random rectangular matrices}
\author{Max Dabagia%
\thanks{School of Computer Science, Georgia Institute of Technology, Atlanta GA~30332, USA. E-mail: {\tt maxdabagia@gatech.edu}.}
\ and Manuel Fernandez V%
	\thanks{School of Mathematics, Georgia Institute of Technology, Atlanta GA~30332, USA. E-mail: {\tt mfernandez39@gatech.edu}.}
}
\date{\today}

\begin{document}

\maketitle 
\begin{abstract}
    Let $A \in \R^{N \times n}$ be a random matrix with independent entries $\xi$ satisfying $\E \xi = 0, \E \xi^2 = 1, \E|\xi|^{2+\al} \leq R$. We show that the smallest singular value $\sigma_n(A)$ satisfies 
    \[
    \Pr\lrpar{\sigma_n(A) \leq \eps(\sqrt{N+1} - \sqrt{n})} \leq (C\eps)^{N-n+1} + e^{-cN},
    \]
    for all $\eps > 0$, where $c,C$ depend only on $\al$ and $R$. This extends earlier results \cite{RV} in the non-square case, where such a lower tail estimate was only known for all aspect ratios when $A$ was i.i.d subgaussian. When the $2+\al$ moment assumption is replaced with a uniform anti-concentration assumption, $\sup_z \Pr\lrpar{|\xi-z| < a} < b$, we show that 
    \[
    \Pr\lrpar{\sigma_n(A) \leq \eps(\sqrt{N+1} - \sqrt{n})} \leq (C\eps\log(1/\eps))^{N-n+1} + e^{-cN},
    \]
    where $c,C$ depend only on $a$ and $b$. This extends more recent work \cite{GL} in the non-square case, where the matrix was required to have i.i.d. rows. To prove these results we overcome a number of difficulties which require new technical ingredients, including a new deviation inequality for the regularized Hilbert-Schmidt norm \eqref{eq:regularized-HS} and a recently proven small ball estimate for the distance between a random vector and a subspace spanned by the columns of an inhomogeneous random rectangular matrix \cite{distance}.
    
\end{abstract}
\section{Introduction}\label{Sec:intro}
Given an $N \times n$ ($N \geq n$) matrix $A$ the singular values of $A$, $\sigma_1(A) \geq \cdots \geq \sigma_n(A)$, are the square roots of the eigenvalues of $A^*A$. In non-asymptotic random matrix theory a problem of fundamental importance has been to understand the behavior of the smallest singular value of a random matrix, 
\[
\sigma_n(A) = \inf_{x \in \bS^{n-1}}\norm{Ax}_2.
\]
For square matrices, sharp estimates for $\sigma_n(A)$ were first obtained in the case where the entries of $A$ are i.i.d. standard gaussian, due to Szarek \cite{szarek} and Edelman \cite{edelman}. Subsequent work by Tao and Vu \cite{TaoVu}, Rudelson \cite{Rudelson1} and finally Rudelson and Vershynin  \cite{RV-square},\cite{RV} established the correct behavior for square random matrices with i.i.d. subgaussian entries. Since then the behavior of the smallest singular value has been extensively studied under various random matrix models. We recommend the recent survey of Tikhomirov \cite{survey} on quantitative invertibility in the non-Hermitian setting for more information on these types of results.
In this paper we consider lower tail estimates for the smallest singular value of non-Hermitian random rectangular matrices. In \cite{Litvak1} Litvak, Rudelson, Pajor, and Tomczak-Jaegermann showed that 
the smallest singular value of an $N \times n$ random matrix with i.i.d. mean zero subgaussian entries with aspect ratio at least $N/n \geq 1 + c_1/(\log n)$ satisfies the following lower tail estimate:
\begin{equation}\label{eq:Litvak1-rectangular}
    \Pr\lrpar{\sigma_n(A) \leq c_2\sqrt{N}} \leq e^{-c_3N}.
\end{equation}
Here $c_1,c_2,c_3$ depend only on the subgaussian moments and the variances of the entries. Besides the result, the paper introduced a number of important ideas for analyzing the smallest singular value, including that of decomposing the unit sphere into compressible and incompressible vectors and applying separate $\epsilon$ net arguments to each set. These ideas were further developed in \cite{RV}, where Rudelson and Vershynin obtained the following lower tail estimate for $\sigma_n(A)$ in the same setting, but now without the aspect ratio assumption:
\begin{equation}\label{eq:RV-rectangular}
    \Pr\lrpar{\sigma_n(A) \leq \eps(\sqrt{N+1} - \sqrt{n})} \leq (C\eps)^{N-n+1} + e^{-cN},
\end{equation}
for all $\eps > 0$.
Of particular note is their development of the `invertibility-via-distance' approach (see \refL{lem:invert-via-dist}) and new estimates for the small ball probability of the distance between a random vector and a subspace (a.k.a. a `distance' theorem). Later in \cite{GL}, Livshyts obtained lower tail estimates for $\sigma_n(A)$ when entries of $A$ are heavy-tailed and non-identically distributed (i.e the inhomogenous case). In particular the following estimate was given when $A$ has identically distributed rows with independent entries that have mean 0, variance 1 and are uniformly anti-concentrated entries:
\begin{equation}\label{eq:GL-rectangular}
    \Pr\lrpar{\sigma_n(A) \leq \eps(\sqrt{N+1} - \sqrt{n})} \leq 
    \begin{cases}
        (C\eps\log(1/\eps))^{N-n+1} + e^{-cN} & \text{ if } N > n, \\
        C\eps + e^{-cN} & \text{ if } N = n,
    \end{cases}
\end{equation}
for all $\eps > 0$. Furthermore in the square case the assumption of mean 0 variance 1 entries can be replaced with an upperbound on the Hilbert-Schmidt norm ($\norm{A}_{HS} := \sqrt{\sum_{i,j}a_{ij}^2}$) of order $n$. The key ingredient in \cite{GL} was the use of new `lattice-like' $\eps$ nets, which allow one to control the distance-approximating guarantee of the net in terms of regularized Hilbert-Schmidt norm, $B_\kappa(A)$ (see \ref{eq:regularized-HS} and \refT{thm:net-full}), instead of the operator norm $\norm{A}$.
Subsequently in \cite{LTV}, Livshyts, Tikhomirov and Vershynin, using a new distance theorem for inhomogenous random matrices, obtained \eqref{eq:GL-rectangular} in the square case without any identical distribution assumptions. 
\par
For the purposes of applications, estimate \eqref{eq:GL-rectangular} applies to a much wider range of random matrix models than i.i.d. subgaussian. However, in contrast to \eqref{eq:RV-rectangular}, the main term in \eqref{eq:GL-rectangular} contains an additional factor of $\log(1/\eps)^{N-n+1}$ in the non-square case. Given the difference in estimates for non-square matrices and the lack of identical distribution assumptions in the square case, it is natural to try and extend \eqref{eq:GL-rectangular} to rectangular matrices without i.i.d. rows and determine conditions sufficient for deducing \eqref{eq:RV-rectangular} in the non-square case. Towards this goal we prove the following two theorems.
\begin{theorem}\label{thm:sv-1}
    Let $a > 0, b \in (0,1)$. There exist positive constants $C,c$, depending on $a$ and $b$ such that the following is true: Let $A \in \R^{N \times n}$ be a random matrix with independent entries $\xi$ that satisfy $\E \xi = 0, \E \xi^2 = 1 , \sup_{z \in \R}\Pr\lrpar{|\xi-z| < a} < b$. Then there exists positive $C,c$ depending on $a$ and $b$ such that
    \[
    \Pr\lrpar{\sigma_n(A) \leq \eps\lrpar{\sqrt{N+1} - \sqrt{n}}} \leq (C\eps\log(1/\eps))^{N-n+1} + e^{-cN}
    \]
\end{theorem}
\begin{theorem}\label{thm:sv-2}
    Let $R \ge 1,\beta > 0$. There exists constants $C,c > 0$, depending on $R$ and $b$ such that the following is true: Let $A \in \R^{N \times n}$ be a random matrix with independent entries $\xi$ that satisfy $\E \xi = 0, \E \xi^2 = 1 , \E |\xi|^{2+\beta} \le R$. Then there exists positive $C,c$ depending on $b$ and $R$ such that
    \[
    \Pr\lrpar{\sigma_n(A) \leq \eps\lrpar{\sqrt{N+1} - \sqrt{n}}} \leq (C\eps)^{N-n+1} + e^{-cN}
    \]
\end{theorem}
In fact our results hold in greater generality (see \refT{thm:sv-tall}, \refT{thm:sv-general1}, \refT{thm:sv-general2}). \textbf{In comparison to previous results} our work improves the state of the art in a number of ways.
Our first theorem, \refT{thm:sv-1}, says that the result of Livshyts for rectangular matrices \cite{GL} holds without any identical distribution assumptions. In particular the assumption about identically distributed rows is not necessary. Our second theorem, \refT{thm:sv-2}, says that the result of Rudelson and Vershynin \cite{RV} holds for matrices with independent entries having mean 0 variance 1 and bounded $2+\al$ moment, for some $\al > 0$. Furthermore once  $N-n$ is of order $n/(\log n)$, a secondary theorem of ours, \refT{thm:sv-tall}, say that the assumptions of mean 0, variance 1 and bounded $2+\al$ moments can be replaced with uniformly anti-concentrated entries and an upperbound on the Hilbert-Schmidt norm of order $Nn$. To our knowledge \refT{thm:sv-2}  is the first such result that recovers \eqref{eq:RV-rectangular} for inhomogenous heavy-tailed matrices under all aspect ratios. We also note that in the context of singular value estimates the use of a $2+\al$ assumption previously appeared in \cite{Litvak2}. In that paper, Litvak and Rivasplata showed that \eqref{eq:Litvak1-rectangular} applied to inhomogenous matrices satisfying assumptions different from ours, even allowing for sparse matrices, but requiring $N-n$ to be of order $n/(\log n)$. In the square case a tail estimate of the form $C(\eps + n^{-\beta/2})$ was given. In particular the intermediate range of aspect ratios was not considered. Since \cite{LTV} and our results imply that the $2+\al$ moment assumption is not necessary to obtain \eqref{eq:RV-rectangular} for inhomogenous heavy-tailed matrices outside the intermediate range, we believe \eqref{eq:RV-rectangular} should hold for such matrices in this range.  
\par
\textbf{Proof strategy and key difficulties:} Recall that our goal is to show that $\inf_{x \in\bS^{n-1}} \norm{Ax}_2$ is unlikely to be small for any unit vector $x$. In the case where the matrix is very tall ($N-n$ is at least of order $n$), it is understood that a lower tail estimate on $\sigma_n(A)$ should be exponentially small. In the case of matrices with i.i.d. entries, Tikhomirov showed in \cite{tall} that estimate $\eqref{eq:Litvak1-rectangular}$ holds so long as $N-n$ is at least of order $n$ and that the entries of $A$ are uniformly anti-concentrated. Under our assumptions Livshyts showed in \cite{GL} an identical estimate when $N-n \geq Mn$ for $M > 0$ sufficiently large. 
The hard case, in fact, is when $A$ is square-like ($1 \leq N-n < \lambda n$ for $\lambda > 0$ small). In this range we would like to follow the approaches of Rudelson and Vershynin \cite{RV} and Livshyts \cite{GL}. This involves decomposing the sphere (see \refD{def:decomposition}) into sets of compressible vectors, $\comp(\delta,\rho)$, and incompressible vectors, $\incomp(\delta,\rho)$, and showing that $\inf_{x \in \comp(\delta,\rho)}\norm{Ax}_2$ and $\inf_{x \in \incomp(\delta,\rho)}\norm{Ax}_2$ are unlikely to be small using certain $\eps$-net arguments. Because we are in the inhomogenous setting we make use of the lattice-like nets that appeared previously in \cite{GL} and \cite{LTV} for analyzing the smallest singular values of inhomogenous random matrices. 

By previous results in \cite{GL}, it is known that $\inf_{x \in \comp(\delta,\rho)}\norm{Ax}_2$ is likely large. Thus the crux of the matter is lower bounding $\inf_{x \in \incomp(\delta,\rho)}\norm{Ax}_2$. When $N-n \leq \lambda n/(\log n)$ we follow the invertibility via distance approach (see \refL{lem:invert-via-dist}). In particular we must show that $\inf_{x \in \spread_d}\norm{W x}_2$ is likely not small, where $W$ is the projection of a certain submatrix of $A$, by showing that $\inf_{y \in \cN} \norm{Wy}_2$ is likely not small, where $\cN$ is a net over $\spread_d$. This requires a way to bound the probability that $\norm{Wy}_2$ is small for a given $y$. In particular we need a distance theorem for inhomogeneous matrices. However, unlike the distance theorem used in \cite{LTV} which allowed for the subspace to be spanned by the columns of an $n \times (n-1)$ inhomogenous random matrix, we require a distance theorem with much larger aspect ratio. Therefore we make use of the distance theorem, \refT{thm:distance}, recently proved in \cite{distance} for inhomogenous matrices with $N-n$ of order at most $N/(\log N)$.
\par 
In conjunction with previously developed tools, \refT{thm:distance} is enough to prove \refT{thm:sv-1}. Obtaining \refT{thm:sv-2}, however, requires further improving the distance-approximating guarantee of lattice-like nets. Because the guarantee of our net is controlled by $\cB_\kappa(W)$ (see \refD{eq:regularized-HS} and \refT{thm:net-full}) we need a sufficiently strong deviation inequality for $\cB_\kappa(W)$. A deviation inequality for $\cB_\kappa(W)$ was given in \cite{GL}, but only allows us to recover \eqref{eq:GL-rectangular}. Consequently, we prove a stronger deviation inequality for matrices whose columns satisfy a certain moment inequality, \refL{lem:hs2pluseps}, and show that the inequality is applicable to $W$ under the assumptions of \refT{thm:sv-2}.    
\par 
In the case where $N-n > \lambda n/(\log n)$ we use a more basic $\eps$-net argument. In particular we take a suitable lattice-like net $\cN$ over $\incomp(\delta,\rho)$, argue that $\inf_{y \in \cN} \norm{Ay}_2$ is large, and use the net properties to deduce that $\inf_{x \in \incomp(\delta,\rho)} \norm{Ax}_2$ is large. We remark that this approach is motivated by the proof of \eqref{eq:Litvak1-rectangular} for sufficiently tall subgaussian matrices in \cite{Litvak1}. 
\par
We now end this section with an outline of the rest of the paper. In \refS{Sec:Pre} we discuss some preliminaries, definitions and notation.
We will also introduce relevant technical results related to anti-concentration and $\eps$-nets. In \refS{Sec:UAC} we prove uniform anti-concentration for random variables satisfying certain moment conditions. In \refS{Sec:moments} we prove lemmas about moments of projected vectors necessary for our implementation of the invertibility via distance approach. In \refS{Sec:tall} we prove smallest singular value estimates in the case where $N-n$ is at least of order $n/(\log n)$. Finally in \refS{Sec:almost square} we prove smallest singular value estimates in the case where $N-n$ is at most of order $n/(\log n)$.
\section*{Acknowledgements}
The authors would like to thank Galyna Livshyts for numerous helpful discussions and feedback. The second author would also like to thank Nicholas Cook and Konstantin Tikhomirov for helpful discussions. The first author was supported by a NSF GRFP.
The second author was partially supported by a Georgia Tech ARC-ACO Fellowship and NSF-BSF DMS-2247834 while working on this paper.   
\section{Preliminaries}\label{Sec:Pre} 
In this paper the parameters $a,\beta > 0, b \in (0,1)$ and $R \ge 1$ are taken to be absolute constants. For $R$ and $\beta$, beyond satisfying their constraints, their choice is arbitrary. In the case of $a$ and $b$ we will either assume that their choice is arbitrary subject to their constraints or that they are some functions of $R$ and $\beta$ (See \refS{Sec:UAC} for details). However, their value/choice is to be fixed throughout the entire paper. We will use $c,c_1,c_2,\cdots,$ etc. to denote a sufficiently small absolute constant that depends on $a,b,R$ and $\beta$ (recall that $a,b,R$ and $\beta$ are absolute constants) and we will use $C,C_1,C_2,\cdots,$etc. to denote a sufficiently large absolute constant that depends on $a,b,R$, and $\beta$. Furthermore $C$ and $c$ may appear multiple times in the same context with different meaning each time, such as in an inequality chain, so as to simplify expressions and collect terms. 
For notation purposes we write $x \lsim y$ to mean $x/y$ is at most a constant depending only on $a,b$, $\beta$ and $R$ and write $x \approx y$ to mean $x \lsim y \lsim x$.  
We say $x \ll y$ if $x/y$ is at most a sufficiently small constant depending only on $a,b$, $\beta$ and $R$. For any matrix $A \in \R^{N \times n}$ we write $\norm{A}_{HS}$ for the Hilbert-Schmidt norm of $A$, which is defined as $\norm{A}_{HS} = \sqrt{\sum_{i,j} a_{ij}^2}$.

We now list some assumptions. We say that a random variable $\xi$ is \textit{normalized} if $\E \xi = 0$ and $\E \xi^2 = 1$, is \textit{uniformly anti-concentrated} if $\sup_{u \in \R} \bP[|u - \xi| \leq a] \leq b$ and satisfies the \textit{higher moment assumption} if $\E |\eta|^{2+\beta} \le R$. We say that a random vector $X \in \R^n$ satisfies the second moment assumption if $\E\norm{Z}_2^2 \ll n^2$. We say that a random matrix $A \in \R^{N \times n}$ satisfies the Hilbert-Schmidt norm assumption if it satisfies $\E\norm{A}_{HS}^2 \lsim Nn$.
We now recall the concept of a decomposition of the sphere into `compressible' and `incompressible' vectors. This notion was introduced by Rudelson and Vershynin in \cite{RV-square}, although a similar idea was used in an earlier work by Litvak, Pajor, Rudelson and Tomczak-Jaegermann\cite{Litvak1}. 
\begin{definition}\label{def:decomposition}
    Let $\delta,\rho \in (0,1)$. We say that a vector $x \in \bS^{n-1}$ is $(\delta,\rho)$-compressible if there there exists a vector $y \in \bS^{n-1}$ with at most $\delta n$ non-zero entries for which 
    \[
    \dist(x,y) \leq \rho
    \]
We write $\comp(\delta,\rho)$ to denote the set of ($\delta,\rho$)-compressible vectors and define $\incomp(\delta,\rho) := \bS^{n-1}\setminus \comp(\delta,\rho)$.
\end{definition}
There are two important properties about the sphere decomposition that we will use. The first property is that incompressible vectors contain a large subset of coordinates that are `spread'.  
\begin{proposition}[Lemma 2.5 \cite{RV}]\label{prop:incomp}
    Let $x \in \incomp(\delta,\rho)$. Then there exists a subset of indices $J \subset \{1,2,\cdots,n\}$ of size $|J| \geq \frac{1}{2}\rho^2\delta n$ such that 
\begin{equation}\label{eq:spread}
    \frac{\rho}{\sqrt{2n}} \leq |x_j| \leq \frac{1}{\sqrt{\delta n}}~\text{ for all } j \in J.
    \end{equation}
\end{proposition}
The second property is that rectangular random matrices send compressible vectors to non-small vectors. In particular compressible vectors avoid the matrix's null space.
 \begin{lemma}[Lemma 5.3 \cite{GL}]\label{lem:comp}
 Let $A \in \R^{N \times n}$ be a random matrix that satisfies the Hilbert-Schmidt norm assumption and has independent entries that are uniformly anti-concentrated. Then  
 \begin{equation}\label{eq:comp}
 \bP\lrpar{\inf_{x \in comp(\delta,\rho)} \norm{Ax}_2 \leq c\sqrt{N  }} \leq e^{-cN}.
 \end{equation}
 where $\delta,\rho \in (0,1)$ and $c > 0$ depend only on $a$ and $b$.
 \end{lemma}
 So as to simplify our arguments we  will also fix the choice of $\delta$ and $\rho$, as guaranteed by \refL{lem:comp}, for the remaining sections.
Next, we introduce some technical results related to anti-concentration. The first such result is the following `tensorization' lemma.
\begin{lemma}[Lemma 2.2 \cite{RV-square}]\label{lem:tensorization}
Let $\xi_1,\cdots,\xi_n$ be a collection of independent random variables satisfying $\bP(|\xi_i| \le \eps) \le K\eps$ for all $\eps > \eps_0 > 0$.
Then 
\[
\bP\lrpar{\sum_{i = 1}^n |\xi_i|^2 \leq \eps^2n} \leq (CK\eps)^n
\]
for all $\eps > \eps_0$, for some $C > 0$.
\end{lemma}
The next few results pertain to anti-concentration of random variables and vectors. The first is a theorem of Rogozin \cite{rogozin}.
\begin{theorem}\label{thm:rogozin}
    Let $v \in \R^n$ be a random vector with independent entries that are uniformly anti-concentrated. Then there exists positive constants $C,c$ depending on $a$ and $b$ such that the following is true: For every $u \in \R^n$ and $\eps > c\norm{u}_\infty$ one has 
    \[
    \sup_{z \in \R}\Pr\lrpar{\abs{\ip{u}{v} - z} < \eps} \leq \frac{C\eps}{\norm{u}_2}.
    \]
\end{theorem}
The second is a lemma of Tikhomirov and Rebrova \cite{rebrova} that is itself a direct consequence of \refT{thm:rogozin}.
\begin{lemma}[Lemma 4.7 \cite{rebrova}]\label{lem:rebrova}
    Let $X \in \R^n$ be a random vector with independent entries that are uniformly anti-concentrated. Then there exists $u > 0$ and $v \in (0,1)$ depending on $a$ and $b$ such that
    \[
    \sup_{x \in \bS^{n-1}}\sup_{z \in \R} \Pr\lrpar{\abs{\ip{X}{x} - z} < u} < v.
    \]
\end{lemma}
The last result is a `distance theorem' for inhomogenous random rectangular matrices. 
\begin{theorem}[Theorem 1\cite{distance}]\label{thm:distance}
Let $X \in \R^n$ be a random vector that satisfies the second moment assumption and has independent entries that are uniformly anti-concentrated. Let $A \in \R^{n \times (n-d)}$ be a
be a random matrix satisfying the Hilbert-Schmidt norm assumption with independent entries that are uniformly anti-concentrated. Then there exists positive constants $C,c,\lambda$ depending on $a$ and $b$ such that the following is true: 
Suppose $1 \leq d \leq \lambda n/\log(n)$. Let $H$ denote the column span of $A$ and $P_{H^\perp}$ be the orthogonal projection operator onto $H^\perp$. Then
\begin{equation}\label{eq:distance}
\sup_{v \in \R^n}\Pr\lrpar{\norm{P_{H^\perp}X - v}_2 \leq t\sqrt{d}} \leq (Ct)^d + e^{-cn}.
\end{equation}
\end{theorem}
We note that in \cite{distance} \refT{thm:distance} was actually proven for $a = 1$. This is equivalent to the above formulation since \eqref{eq:distance} is equivalent to $\sup_{v \in \R^n} \bP(\norm{P_{H^\perp}(X/a) - v/a}_2 \le (t/a)\sqrt{d})$ and replacing $t$ with $t/a$ in the right hand side of \eqref{eq:distance} means replacing $C$ with $C/a$, which is a constant only depending on $a$ and $b$.
We now introduce some results regarding the existence of certain lattice-like $\eps$-nets. Such nets have been used to obtained smallest singular value estimates for random rectangular matrices with i.i.d. rows \cite{GL} and for random square matrices with no i.i.d. assumptions at all \cite{LTV}. We will require two theorems about the existence of such nets. 
The first theorem is a slightly modified version of Theorem 3.9 from \cite{LTV}, which itself is an immediate consequence of Remark 3.6, Theorem 3.7 and Lemma 3.8 of \cite{LTV}.
\begin{theorem}\label{thm:net-basic}\cite{LTV}
    Fix $n \in \N$. Let $\eps \in (0,1/4)$. Then there exists a positive constant $C$  such that for every $N \in \N$ and every random $N \times n$ matrix $A$ with independent columns, with failure probability at most $e^{-2n}$ there exists a net $\Lambda(\eps) \subset \frac{3}{2}B_2^n \setminus \frac{1}{2}B_2^n$ of size $(C/\eps)^n$ such that for every $x \in \bS^{n-1}$ there exists $y \in \Lambda(\eps)$ such that
    \begin{equation}\label{eq:net-basic}
    \norm{x-y}_\infty \leq \frac{\eps}{\sqrt{n}},~\norm{A(x-y)}_2 \leq \frac{\eps\sqrt{2\E\norm{A}_{\text{HS}}^2}}{\sqrt{n}}.
    \end{equation}
\end{theorem} 

The second theorem is a major net theorem of \cite{GL}. 
Given $\kappa > 1$ we define the regularized Hilbert-Schmidt norm of a $N \times n$ matrix $M$, introduced in $\cite{GL}$, as
\begin{equation}\label{eq:regularized-HS}
\cB_\kappa(M) = \min\left\{\sum_{i = 1}^n \alpha_i^2\norm{M_i}_2^2~:~\alpha \in [0, 1]^n, \prod_{i=1}^n \alpha_i \ge \kappa^{-n}\right\}.
\end{equation}
We may then state the following theorem, which asserts the existence of a deterministic net close to the sphere providing an approximation for any deterministic matrix with precision depending on that matrix's regularized Hilbert-Schmidt norm.
\begin{theorem}[Theorem 3 \cite{GL}]\label{thm:net-full}
    Fix $n \in \N$. Let $S \subset \R^n$. Let $\gamma \in (2,\sqrt{n}), \eps \in (0,1/(10\gamma)),\kappa > 1$. Let $N(S,\eps B_2^n)$ denote the minimum number of euclidean balls of radius $\eps$ necessary to cover $S$. Then there exists a net $\cN \subset S + 4\eps \gamma B_2^n$ of size 
    \begin{equation}\label{eq:net-full}
    \size{\cN} =
    \begin{cases}
        N(S,\eps B^n_2) \cdot (C_1\gamma)^{C_2n/\gamma^{0.08}}, &\text{if } \kappa \leq \frac{\log 2}{\gamma^{0.08}},\\
        N(S,\eps B^n_2) \cdot (C\kappa)^n(\log \kappa)^{n-1}, &\text{if } \kappa \geq \frac{\log 2}{\gamma^{0.08}},
    \end{cases}
    \end{equation}
    such that for every $N \in \N$ and for every (deterministic) $N \times n$ matrix $A$ and
    every $x \in \bS^{n-1}$ there exists $y \in \cN$ such that 
    \[
    \norm{A(x-y)}_2 \leq  2\eps\gamma\sqrt{\frac{\cB_\kappa(A)}{n}},
    \]
    where $C,C_1,C_2 > 0$ are absolute constants. 
\end{theorem}

\section{Uniform anti-concentration from moment assumptions}\label{Sec:UAC}
Note that, unlike \refT{thm:sv-1}, \refT{thm:sv-2} does not directly assume uniform anti-concentration of the entries of $A$. Instead one can deduce uniform anti-concentration from the moment assumptions. Note that only having assumptions on the first two moments is not enough. Indeed, the random variable $Z_n := (1-(\sqrt{n}-1)Y)X$ with $X$ a Rademacher random variable and $Y$ an independent Bernoulli with parameter $1/n$ has mean 0 and variance 1 but converges to 0 almost surely. By assuming a bounded $2+\beta$ moment we are able to truncate our random variable while approximately preserving its mean and variance, allowing us to derive uniform anti-concentration.
We begin by considering the normalized case. 
\begin{lemma}\label{lem:uac-moment}
    Let $\xi \in \R$ be a normalized random variable satisfying the higher moment assumption.
    Then $\xi$ satisfies 
    \begin{equation}\label{eq:UAC:1}
    \sup_{z \in \R} \Pr\lrpar{\abs{\xi-z} \leq 1/4} \leq p,
    \end{equation}
    with $p \in (0,1)$ depending only on $\al$ and $R$. 
\end{lemma}

\begin{proof}
We take $m$ to be a positive value that we will specify later. We begin by making note of the following tail-estimates for $\xi$ that follow from Markov's inequality:
\begin{equation}\label{eq:tail-bound}
\begin{split}
\int_{m^2}^\infty \Pr\lrpar{\xi^2 > t}~dt 
&\leq \int_{m^2}^\infty \frac{R}{t^{1+\al/2}}~dt 
= \frac{2R}{\al m^\al}.\\
\int_{m}^\infty \Pr\lrpar{\abs{\xi} > t}~dt 
&\leq \int_{m}^\infty \frac{1}{t^2}~dt 
= \frac{1}{m}.
\end{split}
\end{equation}
Next, we define $\xi_1 := \xi\one_{\abs{\xi} \le m}, \xi_2  := \xi\one_{\abs{x} > m}$. Since $\xi = \xi_1$ whenever $\abs{\xi} \le m$ we have that 
\[
\bP(\abs{\xi - z} \le 1/4) \le \bP(\abs{\xi_1 - z} \le 1/4) + \bP(|\xi| > m) \le \bP(\abs{\xi_1 - z} \le 1/4) + R/m^{2+\beta}.
\]
We now pick $m$ sufficiently large so that $|\E\xi_1 - \E\xi| \le 1/20$ and $|\E\xi_1^2 - \E \xi^2| \le 1/8$. 
We first address the $|\E\xi_1 - \E\xi|$ case. Since $\E\xi = 0, \E \xi_1 = -\E \xi_2$ and $|\E \xi_2| \le \E |\xi_2|$ it suffices to show that $\E|\xi_2| \le 1/20$. Note that  
\[
\E\abs{\xi_2} =\int_0^\infty \Pr\lrpar{\abs{\xi_2} > t}~dt
                \leq m \Pr\lrpar{\abs{\xi} > m} + \int_m^\infty \Pr\lrpar{\abs{\xi} > t}~dt \\
                \leq \frac{1}{m} + \frac{1}{m} = \frac{2}{m},
\]
with the second to last inequality following from the second inequality in \eqref{eq:tail-bound}. Therefore we need that $2/m \le 1/20$ or $m \ge 40$. The argument for $|\E \xi_1^2 - \E \xi^2|$ is similar. Since $\E\xi^2 = \E\xi_1^2 + \E\xi_2^2$ and $\E \xi^2 = 1$ it suffices to show that $\E\xi_2^2 \le 1/8$. Note that
\begin{align*}
    \E \xi_2^2 &= \int_0^\infty \Pr\lrpar{\xi_2^2 > t}~dt,\\
                &= m^2 \Pr\lrpar{\xi_2^2 > m^2} + \int_{m^2}^\infty \Pr\lrpar{\xi_2^2 > t}~dt, \\
                &\leq m^2 \Pr\lrpar{\xi^2 > m^2} + \int_{m^2}^\infty \Pr\lrpar{\xi^2 > t}~dt\\
                &\leq \frac{R}{m^{\al}} + \frac{2R}{\al m^{\al}} = \frac{R(2+\beta)}{\beta m^\beta},
\end{align*}
with the last inequality following from the first inequality in \eqref{eq:tail-bound}. Therefore we need that $R(2+\beta)/(\beta m^\beta) \le 1/8$ or $m \ge (8R(2+\beta)/\beta)^{1/\beta}$. Assuming we have chosen $m$ so as to satisfy these moment constraints, we case on $z$. Write $p_z := \bP(|\xi_1 - z| \le 1/4)$. Suppose $|z| \le 1/2$. Then
 \[
 7/8 \leq \E \xi_1^2 \leq (3/4)^2p_z + m^2(1-p_z) \leq 9p_z/16 + m^2(1-p_z),
 \]
 with the second to last inequality using the fact that 
$|\xi_1| \le m$ almost surely. Isolating for $p$ gives
\[
p_z(m^2 - 9/16) \leq m^2 - 7/8 \implies p_z \leq \frac{m^2 - 7/8}{m^2 - 9/16} = 1 - \frac{5}{16(m^2 -9/16)} \leq 1 - \frac{5}{16m^2},
\]
with the only assumption used being $m > 3/4$. Suppose now that $|z| > 1/2$. Without loss of generality we assume that $z > 1/2$. Note that 
\[
1/20 \ge |\E \xi_1| \ge \E \xi_1 \ge (1/4)p_z + (-m)(1 - p_z),
\]
with the last inequality using the fact that 
$|\xi_1| \le m$ almost surely. Isolating for $p_z$ gives
\[
 p_z(m+1/4) \leq (m + 1/20) \implies p_z \leq \frac{m+1/20}{m+1/4} = 1 - \frac{1/5}{m+1/4} < 1 - \frac{1}{5(m+1)}.
\]
Now, if $m \ge 3$ then $5/(16m^2) \le 1/(5(m+1))$. In addition $R/m^{\beta + 2} \le 1/(16m^2)$ when $m \ge (16R)^{1/\beta}$.
Thus we conclude \eqref{eq:UAC:1} with $p = 1 - 1/(2m)^2$ and $m = \max\{40,(8R(2+\beta)/\beta)^{1/\beta},(16R)^{1/\beta}\}$.
\end{proof}
We now give a slight generalization of \refL{lem:uac-moment}.
\begin{corollary}\label{cor:uac-moment}
    Let $\xi$ be a random variable. Define $\bar{\xi} := \xi - \E \xi$. If $\E\bar{\xi}^2 > 0$ and $\E|\bar{\xi}|^{2+\beta} < \infty$. Then 
    \[
    \sup_{z \in \R} \Pr\lrpar{\abs{\bar{\xi}-z} \leq \sqrt{\E\bar{\xi}^2}/2} \leq p,
    \]
    where $p \in (0,1)$ depends only on $\al$ and $\E|\bar{\xi}|^{2+\beta}/(\E\bar{\xi}^2)^{1+\beta/2}$.
\end{corollary}
\begin{proof}
    Note that $\bar{\xi}/\sqrt{\E\bar{\xi}^2}$ has mean 0, variance 1, and $2+\beta$ moment $\E|\bar{\xi}|^{2+\beta}/(\E\bar{\xi}^2)^{1+\beta/2} < \infty$.
    By \refT{lem:uac-moment} $\bar{\xi}$ satisfies:
    \[
    \sup_{z \in \R} \Pr\lrpar{\abs{\bar{\xi}-z} \leq /\sqrt{\E\bar{\xi}^2}/4} = 
    \sup_{z \in \R} \Pr\lrpar{\abs{\bar{\xi}/\sqrt{\E\bar{\xi}^2}-z} \leq 1/4} \leq b,
    \]
    where $b$ depends only on $\al$ and $\E|\bar{\xi}|^{2+\beta}/(\E\bar{\xi}^2)^{1+\beta/2}$. 
\end{proof}
To conclude this section we note that in order for $\xi$ with mean 0 to be uniformly anti-concentrated it is necessary that $\E|\xi|^2 \gsim 1$. Indeed, assuming uniform anti-concentration one has 
\[
\E|\xi|^2 = \E|\xi - \xi'|^2 \ge a^2\E_{\xi'}\bP_\xi(|\xi - \xi'| \ge a |\xi') \ge a^2\sup_{z \in \R}\bP_\xi(|\xi - z| \ge a) \ge a^2(1-b) \gsim 1,
\]
where $\xi'$ denotes an independent copy of $\xi$.
\section{Moments of projections}\label{Sec:moments}
We now state some results about the moments of the projection of a random vector $X$ onto a fixed subspace. These will be used in analyzing the smallest singular value of almost square matrices. 
The first result says that the norm of a projected random vector with normalized entries scale with the dimension of the subspace it is projected onto.
\begin{lemma}[Lemma 8.3 \cite{GL}]\label{lem:2-moment}
    Let $X \in \R^n$ be a random vector with independent entries $\xi$ that are normalized random variables. Let $H$ be a subspace of co-dimension $d$. Then 
    \[
    \frac{\E\norm{P_{H^\perp} X}_2^2}{\E\norm{X}_2^2} = \frac{d}{n}
    \]
\end{lemma}
The second result says that $\norm{P_{H^\perp}X}_2$ satisfies a reverse Jensen's inequality when the vector's entries do as well.  
\begin{lemma}\label{lem:HS-proj}
Let $X \in \R^n$ be a random vector with independent entries $\xi$ satisfying $\E \xi = 0, \E \xi^2 = 1, \E |\xi|^{2+\beta} \le R$. Let $H$ be a subspace of co-dimension $d$. Then 
\begin{equation}\label{eq:moment:main}
\E\norm{P_{H^\perp} X}_2^{2+\al} \leq Cd^{1+\al/2},
\end{equation}
where $C$ is a positive constant depending only on $R$ and $\al$.
\end{lemma}
The proof of this will require the following inequality due to Rosenthal \cite{rosenthal}.
\begin{lemma}\cite{rosenthal}\label{lem:rosenthal}
    Let $\xi_1,\cdots,\xi_n$ be a set of independent random variables $\xi$ that have mean 0 and finite $p$th moment (i.e. $\E|X_i|^p < \infty$). Then there exists $C > 0$, depending only on $p$, such that 
    \begin{equation}\label{eq:rosenthal}
        \E\left\lvert\sum_{i = 1}^n X_i\right\rvert \leq C \cdot 
        \max \lrpar{\sum_{i = 1}^n \E|X_i|^p,\lrpar{\E\left\lvert
        \sum_{i = 1}^n X_i\right\rvert^2}^{p/2}}
    \end{equation}
\end{lemma}
\begin{proof}[Proof of \refL{lem:HS-proj}]
By convexity the inequality $(z_1+z_2)^p \leq 2^{p-1}(z_1^p+z_2^p)$ holds for all $z_1,z_2 \in \R$ and $p \geq 1$.
More generally given real numbers $z_1,\cdots,z_{2^k}$ on has
\[
\lrpar{\sum_{i = 1}^{2^k} z_i}^p \leq 2^{k(p-1)}\sum_{i = 1}^{2^k} z_i^p.
\]
Let $k$ be the smallest power of two at least as large as $d$. Take $u_1,\cdots,u_d$ to be an orthonormal basis for $H^\perp$ and take $u_{d+1},\cdots,u_{2^k}$ to be copies of the zero vector. Then
\[
\norm{P_{H^\perp}X}^{2+\al}_2 = \lrpar{\sum_{i = 1}^d (u_i^\top X)^2}^{(1+\al/2)} = \lrpar{\sum_{i = 1}^{2^k} (u_i^\top X)^2}^{(1+\al/2)} \leq 
2^{k\al/2}\lrpar{\sum_{i = 1}^{2^k} \abs{u_i^\top X}^{2+\al}} \leq (2d)^{\al/2}\lrpar{\sum_{i = 1}^{d} \abs{u_i^\top X}^{2+\al}}.
\]
Since $u_i^\top X$ is a linear combination of independent mean 0 random variables with bounded $(2+\al)$ moments we may apply \refL{lem:rosenthal} to deduce
\[
\E \abs{u_i^\top X}^{2+\al} = \E \biggabs{\sum_{j = 1}^n u_{ij} X_j}^{2+\al} \leq C\max \lrpar{\sum_{j = 1}^n \E\abs{u_{ij}X_j}^{2+\al}, \lrpar{\E\biggabs{\sum_{j = 1}^n u_{ij}X_j}^2}^{1+\al/2}},
\]
where $C$ depends on $\al$.
The first term in the max can be bounded by 
\[
\sum_{j = 1}^n \E\abs{u_{ij}X_j}^{2+\al} = 
\sum_{j = 1}^n \abs{u_{ij}}^{2+\al}\E\abs{X_j}^{2+\al} 
\le R\sum_{j=1}^n |u_{ij}|^2
= R\]
with the the inequality following from the higher moment assumption and that fact that $u$ is a unit vector. The second term can be bounded by 
\[
\lrpar{\E\biggabs{\sum_{j = 1}^n u_{ij}X_j}^2}^{1+\al/2} = \lrpar{\sum_{j = 1}^n u_{ij}^2\E X_j^2}^{1+\al/2} = 1^{\beta/2} = 1,
\]
with the first two equalities following from the normalization assumption and that $u$ is a unit vector.
Therefore
\[
\E\norm{P_{H^\perp} X}_2^{2+\al} \leq (2d)^{\al/2} \sum_{i = 1}^d \E\abs{u_i^\top X}^{2+\al} \leq (2^{\al/2}CR)\cdot d^{1+\al/2}.
\]
\end{proof}
To conclude this section we give a slight generalization of \refL{lem:HS-proj}.
\begin{corollary}
Let $X \in \R^n$ be a random vectors with independent entries that have identical second moment and finite $(2+\beta)$ moment. Define $\bar{X} := (X - \E X)$. Let $H$ be a subspace of co-dimension $d$. Then 
\[
\E\norm{P_{H^\perp}\bar{X}}_2^{2+\beta} \le C(\E\norm{P_{H^\perp}\bar{X}}_2^2)^{1+\beta/2}
\]
where $C > 0$ depends on $\beta$ and the $2+\beta$ moment. 
\end{corollary}
\begin{proof}
    Applying \refL{lem:HS-proj} to $\bar{X}/((\E\norm{\bar{X}}_2^2/n)^{1/2})$ gives
\begin{align*}
    \E\norm{P_{H^\perp}\bar{X}}_2^{2+\beta} 
    &= (\E\norm{\bar{X}}_2^2/n)^{1+\beta/2} \E \norm{P_{H^\perp}(\bar{X}/(\E\norm{\bar{X}}_2^2/n)^{1/2})}_2^{2+\beta}, \\
    &\le C(\E\norm{\bar{X}}_2^2)^{1+\beta/2}(d/n)^{1+\beta/2}, \\
    &= C(\E\norm{\bar{X}}_2^2)^{1+\beta/2}\lrpar{\frac{\E\norm{P_{H^\perp }\bar{X}}_2^2}{\E\norm{\bar{X}}_2^2}}^{1+\beta/2}, \\
    &= C(\E\norm{P_{H^\perp}\bar{X}}_2^2)^{1+\beta/2},
\end{align*}
where the second to last line follows from \refL{lem:2-moment} and homogeneity of norms.
\end{proof}
\section{Smallest singular value: tall case}\label{Sec:tall}
Having introduced the necessary technical results regarding uniform anti-concentration, moment estimates and discretization, we are ready to prove our estimates for the smallest singular value estimate for inhomogenous random rectangular matrices. We first consider the case where the matrix is tall.

\begin{theorem}\label{thm:sv-tall}
There exists positive constants $C,c,\lambda$ depending on $a$ and $b$ such that the following is true: Let $A \in \R^{N \times n}$ be a random matrix satisfying the Hilbert-Schmidt norm assumption with independent entries that are uniformly anti-concentrated and $N-n \geq \lambda n/\log(n) \geq 1$. Then 
\[
\bP\lrpar{\inf_{x \in \bS^{n-1}} \norm{Ax}_2 \leq \eps\sqrt{N}} \leq (C\eps)^{N-n+1} + e^{-cN}.
\]
\end{theorem}
To prove \refT{thm:sv-tall} we will consider 2 subcases. The first is when $A$ is very tall. This is addressed by the following fact, which follows from Proposition 5.1 from \cite{GL}.
\begin{proposition}\label{prop:very-tall}\cite{GL}
    There exists positive constants $M,C,c$ depending on $a$ and $b$ such that the following is true: Let $A \in \R^{N \times n}$ be a random matrix satisfying the Hilbert-Schmidt norm assumption with independent entries that are uniformly anti-concentrated and $N \geq Mn$. Then 
    \[
    \Pr\lrpar{\inf_{x \in \bS^{n-1}}\norm{Ax}_2 \leq C\sqrt{N}} \leq e^{-cN}.
    \] 
\end{proposition}
The other case is when $\lambda n/\log n \leq N-n < Mn$. 
 For the incompressible vector case we will use a net argument. Our proof is analogous to the proof of smallest singular value estimates for tall subgaussian matrices from \cite{Litvak1}. 
\begin{proposition}\label{prop:incomp-tall}
    There exist positive constants $\lambda,C,c$, depending only on $a$ and $b$ so that the following is true: Suppose $1 \leq \lambda n/\log(n) \leq N-n \leq Mn$, where $M$ is from \refPp{prop:very-tall}. Let $A \in \R^{N \times n}$ be a random matrix satisfying the Hilbert-Schmidt norm assumption with independent entries that are uniformly anti-concentrated. Then
    \begin{equation}\label{eq:incomp-tall}
    \bP\lrpar{\inf_{x \in \incomp(\delta,\rho)}\norm{Ax}_2 \leq \eps\sqrt{N}} \leq (C\eps)^{N-n+1} + e^{-cN},
    \end{equation}
    for all $\eps > 0$.
\end{proposition}
\begin{proof}[Proof of \refT{prop:incomp-tall}]
     Let $\Lambda(\eps)$ be the net from \refT{thm:net-basic} having size at most $(C_1/\eps)^n$ such that, with failure probability at most $e^{-2n}$, every unit vector $x$ has an approximating vector $y \in \Lambda(\eps)$ satisfying 
    \[
    \norm{x-y}_\infty \leq \frac{\eps}{\sqrt{n}},~~\norm{A(x-y)}_2 \le \eps\sqrt{N},
    \]
    with the inequalities following from \eqref{eq:net-basic} and the Hilbert-Schmidt norm assumption.
    Let $\cN \subset \Lambda(\eps)$ consist of all vectors $y$ that approximates some vector in $\incomp(\delta,\rho)$. By \refPp{prop:incomp} there exists $\sigma \subset \{1,\cdots,n\}$ of size $\size{\sigma} \geq \frac{1}{2}\rho^2\delta n \gsim n$ such that 
    \[
    \frac{1}{\sqrt{n}}\lrpar{\frac{\rho}{\sqrt{2}} - \eps}\leq \abs{y_i} \leq \frac{1}{\sqrt{n}}\lrpar{\frac{1}{\sqrt{\delta}} + \eps},
    \]
    for all indices $i$ in $\sigma$.
    If $ \eps \leq \min \{\rho/\sqrt{8},1/\sqrt{\delta}\}$ then
    $\rho/\sqrt{8n} \leq \abs{y_i} \leq 2/\sqrt{\delta n}$
    for all $y \in \cN$. Letting $P_\sigma y \in \R^{\size{\sigma}}$ denote the projection of $y$ onto the coordinate subspace indexed by $\sigma$, we have that 
    \[
    \norm{P_\sigma y}_\infty \leq \frac{2}{\sqrt{\delta n}} \lsim \frac{1}{\sqrt{n}},~~\norm{P_\sigma y}_2 \geq \frac{\rho^2\sqrt{\delta}}{4} \gsim 1.
    \]
    Therefore by \refT{thm:rogozin} we have 
    \[
    \sup_{z \in \R}\bP\lrpar{\abs{\ip{y}{A^\top e_i} - z} < \eps} \leq \sup_{z \in \R}\Pr\lrpar{\abs{\ip{P_\sigma y}{P_\sigma A^\top e_i} - z} < \eps} \lsim \eps,
    \]
    for all $i \in \{1,2,\cdots,N\}$ whenever $\eps \gg n^{-1/2}$. 
    \refL{lem:tensorization} then implies that
    \[
    \Pr\lrpar{\norm{Ay}_2 < 2\eps\sqrt{N}} \leq (C_2\eps)^{N}.
    \]
    Taking a union bound over $\cN$ we get
    \[
    \Pr\lrpar{\inf_{y \in \cN}\norm{Ay}_2 < 2\eps\sqrt{N}} \leq \lrpar{\frac{C_1}{\eps}}^n\lrpar{C_2\eps}^{N} \leq (C\eps)^{N-n}.
    \]
    Observe now that whenever $\inf_{y \in \cN}\norm{Ay}_2 \ge 2\eps\sqrt{N}$ and the distance preserving properties of $\Lambda(\eps)$ hold one has
    $\inf_{x \in \incomp(\delta,\rho)} \norm{Ax}_2 \ge 2\eps\sqrt{N} - \eps\sqrt{N} \ge \eps\sqrt{N}$. Therefore for $n^{-1/2} \ll \eps \le \min\{\rho/\sqrt{8},1/\sqrt{\delta}\}$ we have 
    \[
    \bP\lrpar{\inf_{x \in \incomp(\delta,\rho)}\norm{Ax}_2 < \eps\sqrt{N}} \le (C\eps)^{N-n} + e^{-cn}.
    \]
    On the other hand when $\eps \lsim n^{-1/2}$ we have, since $\lambda n/\log(n) \ge 1$, that $(C\eps)^{N-n} \le e^{-cn}$, with $\lambda$ as small as necessary. Therefore
    \[
    \bP\lrpar{\inf_{x \in \incomp(\delta,\rho)}\norm{Ax}_2 \le \eps\sqrt{N}} \le (C\eps)^{N-n} + e^{-cn}.
    \]
    For all $\eps > 0$, where $C \ge (\min\{\rho/\sqrt{8},1/\sqrt{\delta}\})^{-1}$. To conclude \eqref{eq:incomp-tall} note that we may replace $e^{-cn}$ with $e^{-cN}$ since $n \ge N/(M+1)$ and $M$ is an absolute constant. In addition we may replace $(C\eps)^{N-n}$ with $(C\eps)^{N-n+1}$ since, for $\eps \gsim n^{-1/2}$, one can scale $C$ by some absolute constant $K$ so that $K^{N-n+1} \ge 1/(C\eps)$, hence $(C\eps)^{N-n} \le (CK\eps)^{N-n+1}$.
    
\end{proof}
Therefore \refT{thm:sv-tall} is a corollary of \refPp{prop:very-tall}, \refL{lem:comp} and \refPp{prop:incomp-tall}.
\section{Smallest singular value: almost square case}\label{Sec:almost square}
 Having addressed the tall case, we turn our attention to the almost square case, where $0 \le N-n \leq \lambda n/\log n$. 
\begin{theorem}\label{thm:sv-general1}
    Let $0 \le N-n \leq \lambda n/\log n$. Let $A \in \R^{N \times n}$ be a random matrix satisfying the Hilbert-Schmidt norm assumption with independent entries that satisfy $\E \xi = 0$, are uniformly anti-concentrated and have identical second moments along any given column. Then there exists positive constants $n_0,C,c$ depending only on $a$ and $b$ such that whenever $N \ge n \geq n_0$ and $\eps > 0$ one has
        \[
        \Pr\lrpar{\sigma_n(A) < \eps\lrpar{\sqrt{N+1} - \sqrt{n}}} \leq (C\eps \log(1/\eps))^{N-n+1} + e^{-cN},
        \]
    for all $\eps > 0$.
\end{theorem}

\begin{theorem}\label{thm:sv-general2}
    Let $0 \le N-n \leq \lambda n/\log n$. Let $A \in \R^{N \times n}$ be a random matrix satisfying the Hilbert-Schmidt norm assumption with independent entries $\xi$ that satisfy $\E \xi = 0, \E \xi^2 \gsim 1$, have identical second moments along any given column with $\xi/(\E|\xi|^2)^{1/2}$ satisfying the higher moment assumption. Then there exists positive constants $n_0,C,c$ depending only on $R$ and $\beta$ such that whenever $N \ge n \geq n_0$ and $\eps > 0$ one has
        \[
        \Pr\lrpar{\sigma_n(A) < \eps\lrpar{\sqrt{N+1} - \sqrt{n}}} \leq (C\eps))^{N-n+1} + e^{-cN},
        \]
    for all $\eps > 0$.
\end{theorem}
In conjunction with \refT{thm:sv-tall}, \refT{thm:sv-general1} implies \refT{thm:sv-1} and \refT{thm:sv-general2} implies \refT{thm:sv-2}.
 Here we will present the proof of \refT{thm:sv-general2}. Unlike \refT{thm:sv-general2}, \refT{thm:sv-general1} can be proved in a similar way as Theorem 2 from \cite{GL}, but now using \refT{thm:distance} in lieu of the distance theorem given in \cite{GL} and a slightly modified version of the invertibility-via-distance lemma from \cite{RV}. Moreover our proof of \refT{thm:sv-general2} can be modified in a straightforward way to give a proof of \refT{thm:sv-general1}. Turning our attention to \refT{thm:sv-general2} we note that it suffices to prove a lower bound on $\inf_{x \in \incomp(\delta,\rho)} \norm{Ax}_2$ as the compressible case follows from \refL{lem:comp}. We will follow the invertibility via distance approach implemented in \cite{GL} and \cite{RV}. To that end we begin by reducing the problem of lower bounding $\norm{Ax}_2$ over incompressible vectors to lower bounding the distance between a random vector and a subspace. Let $A_J$ denote the submatrix of $A$ whose columns are indexed by $J$ and let $H_{J^c}$ denote the subspace spanned by the columns of $A$ indexed by $J^c$.
 \begin{definition}\label{def:spread}\cite{RV}
     A vector $v \in \bS^{d-1}$ is said to be spread if all entries satisfy $|v_i| \in [c/\sqrt{d},C/\sqrt{d}]$, for some positive constants $c,C$. We denote the set of spread vectors as $\spread_d$.
 \end{definition}
\begin{lemma}\label{lem:invert-via-dist}
    There exists a positive constant $c$ such that the following is true: Fix $S \subset \{1,2,\cdots,n\}$ with $|S| \le cn$.
    Let $d \in \{1,2\cdots,n -|S|\}$ and $A \in \R^{N \times n}$ be a random matrix. Then there exists $J \subset \{1,2,\cdots,n\} \setminus S$ of size $d$ such that
    \[
    \Pr\lrpar{\inf_{x \in \incomp(\delta,\rho)}\norm{Ax}_2 \leq \frac{\eps d}{\sqrt{n}}} \leq
    C^d\Pr\lrpar{\inf_{x \in \spread_d} \dist(A_Jx,H_{J^c}) \leq \eps\sqrt{d}},
    \]
    where $C$ and the spread parameters in \refD{def:spread} depend only on $\delta$ and $\rho$.
\end{lemma}
The proof of \refL{lem:invert-via-dist} is essentially the same as the proof of the corresponding lemma from \cite{RV} but for completeness we include a proof. 
\begin{proof}
    By \refPp{prop:incomp}  every $x \in \incomp(\delta,\rho)$ has at least $\delta^2\rho n/2$ entries with magnitude in the range $[\rho/\sqrt{2n},1\sqrt{\delta n}]$. Suppose $S$ has at most $c_1n$ elements where  $c_1 = \delta^2\rho/4$.  Then at least $\delta^2\rho n/4$ entries not in $S$ have said magnitudes. Next, let $J$ be distributed uniformly on $\{1,2,\cdots,n\} \setminus S$ and let $\cE(x)$ denote the event that all the entries of $x$ indexed by $J$ have said magnitudes. Then $\bP_J(\cE(x)) \ge \binom{\delta^2\rho n/4}{d}/\binom{n - |S|}{d} \ge c^d$. We now define $D(A,J) = \inf_{x \in \spread_d}\dist(A_Jx_J,H_{J^c})$ and define $\cF := \{A \in \R^{N \times n}:\bP_J((D,A) \ge \eps\sqrt{d}) > 1 - c^d)\}$. Suppose $A \in \cF$.
    Since $\bP_J(\cE(x)) + \bP_J((D,A) \ge \eps\sqrt{d}) > 1$ there exists a realization of $J$ for which $\cE(x)$ holds and $(D,A) \ge \eps\sqrt{d}$. Since the lower bound on $\cE(x)$ holds for all $x \in \incomp(\delta,\rho)$ and $\bP_J((D,A) \ge \eps\sqrt{d})$ is independent of $x$ it follows that for every $x \in \incomp(\delta,\rho)$ there exists some choice of $J$ realizing both events. Since $x_J\sqrt{n/d} \in \spread_d$ it follows that $\norm{Ax}_2 \ge \sqrt{d/n}D(A,J) \ge \eps d/\sqrt{n}$. Thus it suffices to bound $\bP_A(A \not \in \cF)$. Note that
    \begin{align*}
    \bP_A(A \not \in \cF) 
    &= \bP_A\bP_J((D,A) \ge \eps\sqrt{d}) \le 1 - c^d) \\
    &= \bP_A\bP_J((D,A) < \eps\sqrt{d}) > c^d) \\
    &\le c^{-d}\E_A\bP_J((D,A) < \eps\sqrt{d}) \\
    &= c^{-d}\bP_J\E_A((D,A) < \eps\sqrt{d}).
    \end{align*}
For the inequality we used Markov's inequality and in the last equality we used Fubini.
Since $\bP_A(A \not \in \cF) \le c^{-d}\bP_J\E_A((D,A) < \eps\sqrt{d})$ by the Probabilistic Method there exists a choice of $J$ such that $\bP_A(A \not \in \cF) \le c^{-d}\E_A((D,A) < \eps\sqrt{d})$. We conclude as desired. 
\end{proof}
For the remainder of the section we will take $d := N-n+1$, $J$ the subset guaranteed by \refL{lem:invert-via-dist} with $S \subset \{1,2,\cdots\}$ consisting of indices for the $|S| \gsim n$ columns of $A$ with largest second moments. Since $A$ satisfies the Hilbert-Schmidt norm assumption it follows that $\E\norm{A_Je_i}_2^2 \lsim N$ for all $i \in \{1,2,\cdots,d\}$. Consequently $A_J$ also satisfies the Hilbert-Schmidt norm assumption. Furthermore by perturbing the entries of $A$ by independent mean 0 gaussian with arbitrarily small non-zero variance, we may assume that the columns of $A$ are linearly independent almost surely. In particular we may assume that $\dim(H_{J^c}^\perp) = N - (n-d) = 2d-1$. We write $W := P_{H^{\perp}_{J^c}} A_J$ to denote the orthogonal projection of the columns of $A_J$ onto the orthogonal complement of the subspace spanned by the columns of $A_{J^c}$. From \refL{lem:2-moment} and \refT{lem:HS-proj} we conclude the following estimates on $\E\norm{W}_{HS}^2$ and $\E\norm{W}_{HS}^{2+\al}$ respectively :
\begin{proposition}\label{prop:HS-of-W}
\[
    \E\norm{We_i}_{2}^2 = \frac{2d-1}{N}\E\norm{A_Je_i}_2^2
\]
for all $i \in \{1,\cdots,d\}$. In particular 
\[
    \E\norm{W}_{HS}^2 \lsim d^2.
\]
\end{proposition}
\begin{proposition}\label{prop:2+eps-of-W}
\[
    \E\norm{We_i}_{2}^{2+\al} \leq  \lrpar{C\norm{We_i}_2^2}^{1+\al/2},
\]
for all $i \in \{1,\cdots,d\}$, where $C$ depends only on $\al$ and $R$.
\end{proposition}
\par 
\refL{lem:invert-via-dist} reduces the problem of proving \refT{thm:sv-general2} to lower bounding $\inf_{x \in \spread_d}\norm{Wx}_2$. To prove that $\inf_{x \in \spread_d} \norm{Wx}_2$ is large we take an appropriate net $\cN$ over $\spread_d$, as guaranteed by \refT{thm:net-full}, apply \eqref{eq:distance} to all net points and then take a union bound over $\cN$ to deduce that $\inf_{x \in \cN}\norm{Wx}_2$ is large. The distance preserving properties of the net will then imply that $\inf_{x \in \spread_d}\norm{Wx}_2$ is large. 
\subsection{Improved upper tail estimate on $B_\kappa(W)$}

In order to achieve the main term of $(C\eps)^{N-n+1}$ in \refT{thm:sv-general2} as opposed to $(C\eps\log(1/\eps))^{N-n+1}$ in \refT{thm:sv-general1} we will need to prove a sufficiently strong upper tail estimate on the regularized Hilbert-Schmidt norm \eqref{eq:regularized-HS} of $W$, $B_\kappa(W)$, in terms of $\norm{W}_{HS}^2$. In \cite{GL} the following bound was given: 
\[
\Pr\lrpar{B_\kappa(W) \geq C\E\norm{W}^2_{HS}} \leq (c\kappa)^{-2d}.
\]
Here the only assumption on $W$ is that the columns are independent. In the case where the 2-norms of the columns of $W$ satisfy a reverse Jensen's inequality the bound can be further strengthened.

\begin{lemma} \label{lem:hs2pluseps}
    Let $W$ be a random matrix with $d$ columns such that $\E \|W e_i\|_2^2$ exists and moreover for $\al, C > 0$ we have $\E \|W e_i\|_2^{2+\al} \le \lrpar{C \E \|W e_i\|_2^2}^{1+\al/2}$. Then there exists an absolute constant $c$ such that
    \[\Pr\lrpar{\cB_\kappa(W) \ge C\E \|W\|_{HS}^2} \le (c\kappa)^{-(2+\al) d}.\]
\end{lemma}
The proof will require the following proposition, which characterizes the regularized Hilbert-Schmidt norm by taking $y_i = \|We_i\|_2^2$:
\begin{proposition} \label{prop:hsopt}
    Let $y \in [0, \infty)^n$, and $w > 0$. Define $\mathcal S \subseteq 2^{[n]}$ as
    \[\mathcal S = \lrcpar{S \subseteq [n] \colon \forall i \in S, y_i > \lrpar{w \prod_{j \in S} y_j}^{1/\# S}}.\]
    Then
    \[\min_{a} \lrcpar{\sum_{i=1}^n a_i y_i \colon \forall i, a_i \in [0, 1], \prod_{i=1}^n a_i \ge w} = \min_{S} \lrcpar{\sum_{i \not\in S} y_i + \# S \cdot \lrpar{w \prod_{i \in S} y_i}^{1/\# S} \colon S \in \mathcal S}.\]
    Moreover, if $S \in \mathcal S$ achieves the optimum, then
    \[\sum_{i \not\in S} y_i + \# S \cdot \lrpar{w \prod_{i \in S} y_i}^{1/\# S} \le n \cdot \lrpar{w \prod_{i \in S} y_i}^{1/\# S}\]
\end{proposition}

\begin{proof}
    First, if either $y = 0$ or $w = 1$, then $\mathcal S = \{\emptyset\}$ and so the claims hold trivially. So, assume that $y \neq 0$ and $w < 1$, and note that then $\mathcal S$ is not empty, since it at least contains $\{i\}$ for every $y_i > 0$.
    To see that the left minimum is no larger than the right, take $S \in \mathcal S$, and define weights $a \in \R^n$ such that 
    \[a_i = \frac{1}{y_i} \left(w \prod_{i\in S} y_i\right)^{1/\# S}\]
    if $i \in S$, and $1$ otherwise. The definition of $\mathcal S$ then gives $a_i \le 1$ for all $1\le i\le n$. Additionally,
    \[\prod_{i=1}^n a_i = \prod_{i \in S} \lrpar{\frac{1}{y_i} \left(w \prod_{i\in S} y_i\right)^{1/\# S}} = w \frac{\prod_{i \in S} y_i}{\prod_{i \in S} y_i} = w.\]
    Moreover,
    \[\sum_{i=1}^n a_i y_i = \sum_{i \not\in S} y_i + \# S \cdot \left(w \prod_{i\in S} y_i\right)^{1/\# S}.\] so $a$ is a valid weighting.
    Since this holds for every $S \in \mathcal S$, we have
    \[\min\left\{\sum_{i \not\in S} y_i + \# S \cdot \left(w\prod_{i\in S} y_i\right)^{1/\# S} \colon S \in \mathcal S\right\} \ge \min\left\{\sum_{i=1}^n a_i y_i \colon a_i \in [0, 1], \prod_{i=1}^n a_i \ge w\right\}.\]
    To show the other direction, we first formulate the Lagrangian function of the constrained problem:
    \[\mathcal L(a, \mu, \lambda) = \sum_{i=1}^n a_i y_i + \lambda(w - \prod_{i=1}^n a_i) + \sum_{i=1}^n \mu_i(a_i - 1).\]
    The KKT conditions for this problem state that if a feasible vector $a \in \R^n$ is optimal, then there exist $(\lambda, \mu) \ge 0$ such that
    \[\frac{\partial \mathcal L}{\partial a_i}(a, \mu, \lambda) = y_i - \lambda \prod_{j \neq i} a_j + \mu_i = 0\]
    which satisfy complementary slackness: that is, $\lambda(w - \prod_i a_i) = 0$ and $\mu_i(a_i - 1) = 0$ for all $1\le i\le n$. Multiplying the stationarity condition through by $a_i,$ we obtain
    \[(y_i + \mu_i) a_i  = \lambda \prod_{i=1}^n a_i = c\]
    where $c$ is independent of $i$.
    Note that we cannot have $\lambda = 0$, as then $a_i = 0$ for all $i$ which violates the constraint $\prod_{i=1}^n a_i \ge w$. Complementary slackness thus implies that $\prod_{i=1}^n a_i = w$, and additionally that if $a_i < 1$ then we must have $\mu_i = 0$, which means $a_i = c / y_i$. Hence, if we define $S = \{i \colon a_i < 1\}$, it follows that
    \[w = \prod_{i\in S} a_i = \frac{c^{\# S}}{\prod_{i \in S} y_i}\]
    so $c = \left(w\prod_{i\in S} y_i\right)^{1/\# S}$. From the definition of $\mathcal S$, it follows that for every $i \in S$
    \[\frac{1}{y_i} \left(a_i \prod_{i\in S} y_i\right)^{1/\# S} = a_i < 1\]
    so $S \in \mathcal S$, which gives the claimed equality.
    
    To see the inequality at the optimum, let $S \in \mathcal S$, with $a \in \R^n$ the corresponding set of weights as defined above, and suppose there exists $i \not\in S$ such that
    \[y_i > \left(w \prod_{j\in S} y_j\right)^{1/\# S}.\]
    Let $y_j$ be the smallest element in $S$. Note that $a_j < 1$, and
    \[a_j y_j = \left(w \prod_{j\in S} y_j\right)^{1/\# S} < y_i.\]
    The plan is to show that by adjusting the weights of only $y_i$ and $y_j$, we can decrease the objective, which implies $S$ is not optimal. There are two cases: If $y_j > \sqrt{a_j y_i y_j}$, then define $\tilde a \in \R^n$ as 
    \[\tilde a_k = \begin{cases}
        a_k & k\neq i, j\\
        \frac{\sqrt{a_j y_i y_j}}{y_k} &k=i, j
    \end{cases}.\]
    This is a valid weighting since by assumption $\tilde a_j < 1$ and likewise $\tilde a_i = \frac{\sqrt{a_j y_i y_j}}{y_i} < \frac{y_i}{y_i} = 1$, and furthermore
    \[\prod_{k=1}^n \tilde a_k = \frac{a_j y_i y_j}{y_i y_j} \prod_{k \neq i, j} a_k = \prod_{k=1}^n a_k = w.\] Moreover,
    \[\frac{\sqrt{a_j y_i y_j}}{y_i} \cdot y_i + \frac{\sqrt{a_j y_i y_j}}{y_j} \cdot y_j = 2\sqrt{a_j y_i y_j} < y_i + a_j y_j\]
    by the AM-GM inequality, so $\sum_{k=1}^n \tilde a_k y_k \le \sum_{k=1}^n a_k y_k $. The second case is if $y_j \le \sqrt{a_j y_i y_j}$. Now define $\tilde a_i = a_j$ and $\tilde a_j = 1$. Clearly $\tilde a$ remains feasible, and observe that
    \[y_i + a_j y_j - (a_j y_i + y_j) = (1-a_j) (y_i - y_j) > 0 \]
    as $1 > a_j$ and $y_i > a_j y_i \ge y_j$ by assumption, so again $\sum_{k=1}^n \tilde a_k y_k \le \sum_{k=1}^n a_k y_k $. In either case, $a$ cannot be optimal. Hence, if $S$ is the set which achieves the optimum, we must have
    \[y_i \le \left(w \prod_{j\in S} y_j\right)^{1/\# S}\]
    for all $i \not \in S$,
    which gives the inequality.
\end{proof}

\begin{proof}[Proof of Lemma \ref{lem:hs2pluseps}]
    Let $\mathbf S$ be the (random) set defined as
    \[\mathbf S = \arg\min_{S \subseteq [d]} \left\{\sum_{i \not \in S} \|W e_i\|_2^2 + \frac{\# S}{\kappa^{2d/\# S}} \left(\prod_{i\in S} \|W e_i\|_2^2\right)^{1/\# S} \colon \forall i \in S, \|W e_i\|_2^2 > \frac{\left(\prod_{j\in S} \|W e_j\|_2^2\right)^{1/\# S}}{\kappa^{2d/\# S}}\right\}.\]
    Define the random variable
    \[Y(\mathbf S) = \frac{d}{\kappa^{2d/\# \mathbf S}} \left(\prod_{i\in \mathbf S} \|W e_i\|_2^2\right)^{1/\#\mathbf S}.\]
    Applying the proposition to the vector $(\|W e_1\|_2^2, \ldots, \|W e_d\|_2^2)$ and weight $\kappa^{-2d}$, we obtain
    \[\cB_\kappa(W) = \sum_{i \not \in \mathbf S} \|W e_i\|_2^2 + \frac{\# \mathbf S}{\kappa^{2d/\# \mathbf S}} \left(\prod_{i\in \mathbf S} \|W e_i\|_2^2\right)^{1/\# \mathbf S} \le Y(\mathbf S).\] 
    Hence,
    \begin{align*}
        \Pr(Y(\mathbf S) \ge C \E \|W\|_{HS}^2) = \sum_{S \subseteq [n]} \Pr(Y(S) \ge C \E \|W\|_{HS}^2, \mathbf S = S) \le \sum_{S \subseteq [n]} \Pr(Y(S) \ge C \E \|W\|_{HS}^2).
    \end{align*}
    The terms of the sum can be bounded using Markov's inequality as
    \begin{align*}
        \Pr(Y(S) \ge C \E \|W\|_{HS}^2) &= \Pr \lrpar{\frac{d}{\kappa^{2d/\# S}} \left(\prod_{i\in S} \|W e_i\|_2^2\right)^{1/\# S} \ge C \E \|W\|_{HS}^2}\\
        &\le \Pr \lrpar{\prod_{i\in S} \|W e_i\|_2^{2+\al} \ge \kappa^{(2+\al) d} \lrpar{\frac{C}{d} \sum_{i \in S} \E \|W e_i\|_2^2}^{(1+\al/2) \# S} }\\
        & \le \frac{\E \lrsqpar{\prod_{i \in S} \|W e_i\|_2^{2+\al}}}{\kappa^{(2+\al) d} \lrpar{\frac{C}{d} \sum_{i \in S} \E \|W e_i\|_2^2}^{(1+\al/2) \# S}}
    \end{align*}
    By independence, and using the assumed bounds on the moments of the column norms followed by AM-GM,
    \[\E \prod_{i \in S} \|W e_i\|_2^{2+\al} \le \prod_{i \in S} \E \|W e_i\|_2^{2+\al} \le \prod_{i \in S} (C \E \|W e_i\|_2^2)^{1+\al/2} \le \lrpar{\frac{C}{\# S} \sum_{i \in S} \E \|W e_i\|_2^2}^{(1 + \al/2)\# S} \]
    which gives
    \[\Pr(\cB_\kappa(W) \ge C \E \|W\|_{HS}^2) \le \frac{1}{\kappa^{(2+\al)d}} \sum_{S \subseteq [n]} \lrpar{\frac{d}{\# S}}^{(1 + \al/2)\#S}.\]
    Since $(d / \# S)^{\# S} \le e^{d/e}$, we conclude that
    \[\Pr(\cB_\kappa(W) \ge C \E \|W\|_{HS}^2) \le \frac{2^de^{(1+\al/2)d/e}}{\kappa^{(2+\al)d}}.\]
\end{proof}
Note that the columns of $W := P_{H_{J^c}^\perp}A_j$ satisfy the assumptions of \refL{lem:hs2pluseps}. 
\subsection{Net argument}
Recall that we would like to implement a net argument to prove that $\Pr\lrpar{\inf_{x \in \spread_d}\norm{Wx}_2 \leq \eps\sqrt{d}}$ is small. The naive approach would be to first bound the probability that $\cB_\kappa(W) \gg d^2$ and then, using the net guaranteed by \refT{thm:net-full} and a union bound, bound the probability that $\inf_{x \in \spread_d} \norm{Wx}_2 \le \eps\sqrt{d}$ and $\cB_\kappa(W) \lsim d^2$. However, as shown in both \cite{RV} and \cite{GL} for the case of subgaussian matrices and heavy-tailed matrices respectively, this approach will give a suboptimal probability estimate. Instead we follow the `decoupling and iteration' approach employed in \cite{RV} and \cite{GL}. The starting point of such an approach is to bound $\Pr\lrpar{\norm{Wy}_2 \leq \eps\sqrt{d}}$ for every net point $y$. 
\begin{proposition}
\label{prop:net-bound-1}
Let $w \in \R^n$ and $x \in \frac{3}{2}B_2^n \setminus \frac{1}{2}B_2^n$. Then there exist positive constants $C,c$ depending on $R$ and $\beta$ such that  
\begin{equation}\label{eq:net-bound-1}
     \Pr\lrpar{\norm{Wx - w}_2 \leq t\sqrt{d}} \leq (Ct)^{2d-1} + e^{-cN} 
\end{equation}
\end{proposition}
\begin{proof}
    Recall that for every entry $\xi$ of $A_J$ the random variable $\xi/(\E|\xi|^2)^{1/2}$ is normalized and satisfies the higher moment assumption. Since $\E|\xi^2| \gsim 1$ by \refC{cor:uac-moment} the entries are also uniformly anti-concentrated, with $a,b$ depending on $R$ and $\al$. Since the entries of $A_J$ are independent and uniformly anti-concentrated and $\norm{x}_2 \in [1/2,3/2]$ by \refL{lem:rebrova} $A_Jx$ is a random vector with independent entries that are uniformly anti-concentrated, with $a$ and $b$ depending on $R$ and $\al$. In addition $\E\norm{A_Jx}_2^2 \lsim n \ll n^2$, so $A_Jx$ satisfies the second moment assumption. Since $W = P_{H_{J^c}^\perp}A_J$ and $H_{J^c}^\perp$ is independent of $A_J$ we may apply \refT{thm:distance} to conclude that 
    \[
    \Pr\lrpar{\norm{Wx-w}_2 \leq t\sqrt{d}} \leq (Ct)^{2d-1} + e^{-cN}.
    \]
\end{proof}
The next step is to `decouple' the events that $\norm{Wx}_2$ is small and $\cB_\kappa(W)$ is large so as to obtain a stronger probability estimate on the joint event. For our setting we employ the following decoupling lemma from \cite{GL}.
\begin{lemma}[Lemma 8.8 \cite{GL}]\label{lem:decouple}
    Let $d \geq 2$. Let $W$ be an $N \times d$ random matrix with independent columns. Let $z \in \frac{3}{2}B_2^d \setminus \frac{1}{2}B_2^d$ be such that $\abs{z_k} \geq c_1/\sqrt{d}$ for some absolute constant $c_1$. Then for every $0 < a < b$ we have 
    \[
    \Pr\lrpar{\norm{Wz}_2 < u, \cB_{\kappa^2}(W) > v} \leq 2 \Pr\lrpar{\cB_\kappa(W) > \frac{v}{2}} \sup_{x \in \frac{3}{2}B_2^d \setminus \frac{1}{2}B_2^d, w \in \R^N} \Pr\lrpar{\norm{Wx - w}_2 < c_2 u},
    \]
    for some absolute constant $c_2$.
\end{lemma}
With these two facts we are ready to implement the net argument.
\begin{proposition}\label{prop:net-arg}
There exists positive constants $C,c$ depending on $R$ and $\beta$ such that
\begin{equation}\label{eq:net-arg}
    \Pr\lrpar{\inf_{x \in \spread_d}\norm{Wx}_2 \leq \eps\sqrt{d}} \leq (C\eps)^d + e^{-cN}
\end{equation}
\end{proposition}
\begin{proof}
We first note that, since the left hand side of \eqref{eq:net-arg} is non-decreasing in $\eps$, it suffices to consider the case where $\log(1/\eps)$ is at most of order $N/d$. In particular we may assume that $\log(1/\eps) \ll N/d$ so that, plugging in $\eps$ for $t$ in the right hand side of \eqref{eq:net-bound-1} of \refPp{prop:net-bound-1}, one has $(C\eps)^{2d-1} + e^{-cN} \lsim (C\eps)^{2d-1}$.  
We begin by recording a well known fact (see proposition 2.1 in \cite{RV}) that the covering number of $\bS^{d-1}$ by unit balls of radius $\eps$, $N(\bS^{d-1},\eps B_2^d)$, is at most $(C_1/\eps)^{d-1}$. Next, we define $\cN_i$ to be the net of size $(C_2i/\eps)^{d-1}e^{di}$ guaranteed by \refT{thm:net-full} with $\gamma = 3$ such that $\sup_{x \in \bS^{d-1}}\inf_{y \in \cN}\norm{A(x-y)}_2 \le \eps\sqrt{\cB_\kappa(W)}/(\sqrt{2C_3d})$. Taking $m := \ceil{\log(C_4\exp(N/d))}$ 
we define the following set of events:
\begin{align*}
    \cE_1 &:= \lrset{\inf_{x \in \spread_d}\norm{Wx}_2 \leq \eps\sqrt{d}, \cB_{e}(W) < 2C_3d^2}, \\
    \cE_{2,i} &:= \lrset{\inf_{x \in \spread_d}\norm{Wx}_2 \leq \eps\sqrt{d},\cB_{e^{i+1}}(W) < 2C_3d^2 \leq \cB_{e^i}(W)},\\
    \cE_3 &:= \lrset{\cB_m(W) \geq 2C_3d^2},
\end{align*}
where $i$ ranges from 1 to $m-1$.
One can check that the event $\lrset{\inf_{x \in \spread_d} \norm{Wx}_2 \leq s\sqrt{d}}$ is contained in $\cE_1 \cup \lrpar{\cup_{1 \leq i \leq m} \cE_{2,i}} \cup \cE_3$.
Therefore
\begin{align*}
\Pr\lrpar{\inf_{x \in \spread_d}\norm{Wx}_2 \leq s\sqrt{d}} \leq 
\Pr\lrpar{\cE_1 \cup \lrpar{\cup_{1 \leq i \leq m}\cE_{2,i}}  \cup \cE_3} \leq \Pr(\cE_1) + \Pr(\cE_3) + \sum_{i = 0}^m \Pr(\cE_{2,i}).
\end{align*}
We first bound $\bP(\cE_3)$. By \refPp{prop:HS-of-W} $\E\norm{W}_{HS}^2 \lsim d^2$. Therefore by \refL{lem:hs2pluseps} we have that
\[
\bP(\cE_3) \le (cC_4\exp(N/d))^{-(2+d)\beta} = e^{-cN}.
\]
Next we bound $\bP(\cE_1)$. Note that if $\cE_1$ occurs then $\inf_{y \in \cN_1}\norm{Ay}_2 \le \eps\sqrt{d} + \sqrt{\cB_e(W)}/(\sqrt{2C_3d}) \le 2\eps\sqrt{d}$.  Therefore 
\begin{align*}
\bP(\cE_1) 
&\le \bP\lrpar{\inf_{y \in \cN_1}\norm{Wx}_2 \le 2\eps\sqrt{d}, \cB_e(W) < 2C_3d^2}, 
\\
&\le \sum_{y \in \cN_1} \bP(\norm{Wy}_2 \le 2\eps\sqrt{d}),
\\
& \le (C_2/\eps)^{d-1}e^d((C\eps)^{2d-1}), \\
& \le (C\eps)^d.
\end{align*}
In the second to last inequality we used \refPp{prop:net-bound-1}.
We now bound $\sum_{2 \le i < m}\bP(\cE_{2,i})$. Note that if $\cE_{2,i}$ occurs then $\inf_{y \in \cN_i}\norm{Ay}_2 \le \eps\sqrt{d} + \sqrt{\cB_{e^{i+1}}(W)}/(\sqrt{2C_3d}) \le 2\eps\sqrt{d}$. Therefore 
\begin{align*}
\sum_{2 \le i < m}\bP(\cE_{2.i}) 
&\le \sum_{2 \le i < m}\bP\lrpar{\inf_{y \in \cN_i}\norm{Wx}_2 \le 2\eps\sqrt{d}, \cB_{e^{i+1}}(W) < 2C_3d^2 \le \cB_{e^i}(W)} 
\\
&\le \sum_{2 \le i < m} \sum_{y \in \cN_i} \bP\lrpar{\norm{Wy}_2 \le 2\eps\sqrt{d}, \cB_{e^i}(W) \ge 2C_3d^2} \\
&\le \sum_{2 \le i < m} \sum_{y \in \cN_i} 2\bP\lrpar{\cB_{e^{i/2}}(W) \ge 2C_3d^2} \sup_{\underset{x \in \frac{3}{2}B_2^d\setminus \frac{3}{2}B_2^d}{w \in \R^d}}\bP\lrpar{\norm{Wx - w}_2 \le 2c\eps\sqrt{d}} \\
&\le \sum_{2 \le i < m} (C_2i/\eps)^{d-1}e^{di}(ce^{i/2})^{-(2+\beta)d}(C\eps)^{2d-1} \\
&\le (C\eps)^d\sum_{2 \le i < m}i^{d-1}e^{-i\beta d/2} \\
&\le (C\eps)^d.
\end{align*}
In the third inequality we used \refL{lem:decouple}, in the third fourth inequality we used \refPp{prop:net-bound-1}, and in the fifth inequality we used the fact that $\sum_{2 \le i < \infty} \sum i^{d-1} e^{-\beta d/2} \le C^d$, where $C$ depends on $\beta$. Therefore $\bP(\cE_1) + \sum_{2 \le i < m}\bP(\cE_{2,i}) + \bP(\cE_3) \le (C\eps)^d$ and we conclude as desired.
\end{proof}
We now prove \refT{thm:sv-general2}
\begin{proof}[Proof of \refT{thm:sv-general2}]
Since $d = N+1-n$ we have that
\begin{align*}
\Pr\lrpar{\sigma_n(A) \leq \eps(\sqrt{N+1} - \sqrt{n})}  &\le
\Pr\lrpar{\sigma_n(A) \le \frac{\eps d}{\sqrt{n}}} \\
&\le C^d\bP\lrpar{\inf_{x \in \spread_d} \norm{Wx}_2 \le \eps\sqrt{d}}. \\
&\le C^d((C\eps)^d + e^{-cN}), \\
&\le (C\eps)^d + e^{-cN}.
\end{align*}
In the first inequality we used \refL{lem:invert-via-dist} and in the second inequality we used \refPp{prop:net-arg}.
\end{proof}
\bibliographystyle{plain}

\begin{thebibliography}{10}

\bibitem{edelman}
Alan Edelman.
\newblock Eigenvalues and condition numbers of random matrices.
\newblock {\em SIAM journal on matrix analysis and applications},
  9(4):543--560, 1988.

\bibitem{Litvak2}
Alexander Litvak and Omar Rivasplata.
\newblock Smallest singular value of sparse random matrices.
\newblock {\em arXiv preprint arXiv:1106.0938}, 2011.

\bibitem{Litvak1}
Alexander~E. Litvak, Alalin Pajor, Mark Rudelson, and Nicole
  Tomczak-Jaegermann.
\newblock Smallest singular value of random matrices and geometry of random
  polytopes.
\newblock {\em Advances in Mathematics}, 195:491--523, 2005.

\bibitem{GL}
Galyna~V Livshyts.
\newblock The smallest singular value of heavy-tailed not necessarily iid
  random matrices via random rounding.
\newblock {\em Journal d'Analyse Math{\'e}matique}, 145(1):257--306, 2021.

\bibitem{LTV}
Galyna~V. Livshyts, Konstantin Tikhomirov, and Roman Vershynin.
\newblock The smallest singular value of inhomogeneous square random matrices,
  2019.

\bibitem{rebrova}
Elizaveta Rebrova and Konstantin Tikhomirov.
\newblock Coverings of random ellipsoids, and invertibility of matrices with
  iid heavy-tailed entries.
\newblock {\em Israel Journal of Mathematics}, 227:507--544, 2018.

\bibitem{rogozin}
BA~Rogozin.
\newblock Estimation of the maximum of a convolution of bounded densities.
\newblock {\em Theory of Probability \& Its Applications}, 32(1):48--56, 1988.

\bibitem{rosenthal}
Haskell~P Rosenthal.
\newblock On the subspaces of l p (p> 2) spanned by sequences of independent
  random variables.
\newblock {\em Israel Journal of Mathematics}, 8:273--303, 1970.

\bibitem{Rudelson1}
Mark Rudelson.
\newblock Invertibility of random matrices: norm of the inverse.
\newblock {\em Annals of Mathematics}, pages 575--600, 2008.

\bibitem{RV-square}
Mark Rudelson and Roman Vershynin.
\newblock The littlewood--offord problem and invertibility of random matrices.
\newblock {\em Advances in Mathematics}, 218(2):600--633, 2008.

\bibitem{RV}
Mark Rudelson and Roman Vershynin.
\newblock The smallest singular value of a random rectangular matrix.
\newblock 2008.

\bibitem{szarek}
Stanislaw~J Szarek.
\newblock Condition numbers of random matrices.
\newblock {\em Journal of Complexity}, 7(2):131--149, 1991.

\bibitem{TaoVu}
Terence Tao and Van~H Vu.
\newblock Inverse littlewood-offord theorems and the condition number of random
  discrete matrices.
\newblock {\em Annals of Mathematics}, pages 595--632, 2009.

\bibitem{survey}
Konstantin Tikhomirov.
\newblock Quantitative invertibility of non-hermitian random matrices.
\newblock {\em arXiv preprint arXiv:2206.00601}, 2022.

\bibitem{tall}
Konstantin~E Tikhomirov.
\newblock The smallest singular value of random rectangular matrices with no
  moment assumptions on entries.
\newblock {\em Israel Journal of Mathematics}, 212:289--314, 2016.

\bibitem{distance}
Manuel~Fernandez V.
\newblock A distance theorem for inhomogenous random rectangular matrices,
  2024.

\end{thebibliography}

\end{document}